\documentclass[11pt]{amsart}
\usepackage{amsmath,amscd,amssymb,latexsym,color}
\usepackage{tikz-cd, bbm}
\usepackage{longtable}
\usepackage{systeme}
\usepackage{bbm}

\newcommand{\sm}{\left(\smallmatrix}
\newcommand{\esm}{\endsmallmatrix\right)}
\newcommand{\mat}{\begin{pmatrix}}
\newcommand{\emat}{\end{pmatrix}}
\renewcommand{\c}{\mathfrak{c}}
\renewcommand{\t}{\tau}

\renewcommand{\a}{\alpha}

\renewcommand{\i}{\infty}
\newcommand{\G}{\Gamma}
\newcommand{\g}{\gamma}

\renewcommand{\i}{\infty}
\newcommand{\lt}{\left}
\newcommand{\rt}{\right}
\newcommand{\Q}{\mathbb Q}
\newcommand{\Z}{\mathbb Z}
\newcommand{\C}{\mathbb C}

\newcommand{\J}{\mathbb J}
\renewcommand{\H}{\mathbb H}

\newcommand{\MB}{\mathcal B}
\newcommand{\ka}{\kappa}

\newtheorem{thm}{Theorem}

\newtheorem{cor}[thm]{Corollary}
\newtheorem{prop}[thm]{Proposition}

\theoremstyle{remark}

\numberwithin{equation}{section}
\numberwithin{thm}{section}
\begin{document}

\title[Hecke Equivariance of Divisor Lifting involving Sesquiharmonic Maass Forms]{Hecke Equivariance of Divisor Lifting with respect to Sesquiharmonic Maass Forms}

\author{Daeyeol Jeon, Soon-Yi Kang and Chang Heon Kim}
\address{Department of Mathematics Education, Kongju National University, Kongju, 314-701, Republic of Korea}
\email{dyjeon@kongju.ac.kr}
\address{Department of Mathematics, Kangwon National University, Chuncheon, 200-701, Republic of Korea} \email{sy2kang@kangwon.ac.kr}
\address{Department of Mathematics, Sungkyunkwan University, Suwon, 16419 Republic of Korea\\
School of Mathematics, Korea Institute for Advanced Study, 85 Hoegiro, Dongdaemun-gu, Seoul 02455, Republic of Korea}
\email{chhkim@skku.edu}

\begin{abstract} We investigate the properties of Hecke operators for sesquiharmonic Maass forms.
We begin by proving the Hecke equivariance of the divisor lifting with respect to sesquiharmonic Maass functions, which maps an integral weight meromorphic modular form to the holomorphic part of the Fourier expansion of a weight~$2$ sesquiharmonic Maass form.
Using this Hecke equivariance, we show that the sesquiharmonic Maass functions whose images under the hyperbolic Laplace operator are the Faber polynomials $J_n$ of the $j$-function form a Hecke system analogous to $J_n$.
By combining the Hecke equivariance of the divisor lifting with that of the Borcherds isomorphism, we extend Matsusaka's results on the twisted traces of sesquiharmonic Maass functions.
\end{abstract}
\maketitle
\renewcommand{\thefootnote}%
             {}
 \footnotetext{
\par
The first author was supported by Basic Science Research Program through the National Research Foundation of Korea (NRF) funded by the Ministry of Education (2022R1A2C1010487),
the second author was supported by the National Research Foundation of Korea (NRF) funded by the Ministry of Education (NRF-2022R1A2C1007188, RS-2025-25415913)
and
the third author was supported by the National Research Foundation of Korea(NRF) grant 
funded by the Korea government(MSIT)(RS-2024-00348504).
}

\section{Introduction}
The modular $j$-function defined for $\t=x+iy$ in the complex upper half-plane $\H$ by
$$
j(\tau):=q^{-1}+744+196884q+21493760q^2+\cdots \quad (q=e^{2\pi i \t})
$$
holds great significance in number theory and the theory of modular forms.
It is a key object in the study of modular functions, serving as a generator for all meromorphic modular functions for $\G:=\mathrm{SL}_2(\mathbb{Z})$ and playing a crucial role in the theory of elliptic curves and complex multiplication.
Meanwhile, Hecke operators play an important role in the structure of vector spaces of modular forms.
They provide a means to control these spaces using tools from linear algebra, offer insights into their arithmetic properties, and help in understanding the connections between modular forms and representation theory or other areas of mathematics.

For a positive integer $n$, let $J_n(\t):=q^{-n}+O(q)$ be the unique weakly holomorphic modular function for which $J_n(\t)-q^{-n}$ has a $q$-expansion with only strictly positive powers of $q$.
Set $J_0(\t)=1$. Then
\[
J_1(\t) = j(\t) - 744
\]
is the normalized Hauptmodul for $\G$, and $J_n\ (n\geq 0)$ form a basis for the space of weakly holomorphic modular functions for $\G$.
Moreover, they form a Hecke system. In other words, for the normalized Hecke operator $T_n$, they satisfy
\begin{equation}\label{hs}
J_n(\t) = J_1(\t) \mid T_n \quad (n\ge 1).
\end{equation}
Using \eqref{hs}, Asai, Kaneko, and Ninomiya~\cite[Theorem 3]{AKN} proved that the generating function of $J_n$ is a weight~$2$ meromorphic modular form. Namely, for every $\t, z\in\H$,
\begin{equation}\label{akn}
\sum_{n=0}^\infty J_{n}(z) q^n=-\frac{1}{2\pi i}\frac{j'(\t)}{j(\t)-j(z)}.
\end{equation}
This identity is famous for being equivalent to the denominator formula for the Monster Lie algebra.

Bruinier, Kohnen, and Ono~\cite{BKO} extended \eqref{akn} to show that for any meromorphic modular function $f$, the sum of the generating functions of $J_n$ at each divisor of $f$ in the fundamental domain of $\G$ is the negative of the logarithmic derivative of $f$, a weight~$2$ meromorphic modular form on $\G$.
More precisely, if we let
\begin{equation}\label{f}
f(\t)=q^h+\sum_{n=h+1}^\infty a_f(n)q^n
\end{equation}
be a non-zero weight~$k$ meromorphic modular form on $\G$
and let $E_2(\t):=1-24\sum_{n=1}^\infty \sigma(n)q^n$
be the weight~$2$ Eisenstein series, a quasimodular form of weight~$2$, where $\sigma(n):=\sum_{d\mid n} d$.
Then \cite[Theorem 1]{BKO} states that
\begin{equation}\label{divmf}
D(f):=\sum_{n=0}^\infty \sum_{z\in \G\backslash \mathbb{H}}
\frac{\hbox{ord}_{z}(f)}{\omega_{z}} J_{n}(z) q^n
=-\frac{\Theta(f)}{f}+\frac{kE_2}{12},
\end{equation}
where $\Theta:=q\frac{d}{dq}=\frac{1}{2\pi i}\frac{d}{d\t}$ is Ramanujan's differential operator, $\omega_{z}$ is the cardinality of the stabilizer of $z\in \H$ in $\mathrm{PSL}_2(\mathbb{Z})$, and $\hbox{ord}_{z}(f)$ is the order of vanishing of $f$ at $z$.

The divisor lifting involving the generating function of $J_n$ for a weight~$k$ meromorphic modular form to a weight~$2$ meromorphic modular form, as described in \eqref{divmf}, was generalized by Bringmann et al.~\cite[Theorem 1.3]{BKLOR} and Choi, Lee, and Lim~\cite[Theorem 1.1]{CLL}.
They extended this lifting to meromorphic modular forms of arbitrary level with respect to Niebur-Poincar\'e harmonic weak Maass functions, resulting in the holomorphic part of the Fourier expansion of a weight~$2$ polar harmonic weak Maass form.

In \cite{JKK-hequiv}, the authors demonstrated that the divisor lifting is Hecke equivariant under the multiplicative Hecke operator on integral weight meromorphic modular forms and the Hecke operator on weight~$2$ meromorphic modular forms.
They also proved that the generalized Borcherds lifting operator is Hecke equivariant under the multiplicative Hecke operator and the Hecke operator on half-integral weight vector-valued harmonic weak Maass forms.
As applications of the Hecke equivariance of the two operators, they obtained relations for twisted traces of singular moduli modulo prime powers and congruences for twisted class numbers modulo primes.
In their previous work~\cite{JKK-alg}, they defined a lifting with respect to the canonical basis of the space of weakly holomorphic modular functions of arbitrary level~$N$.
This lifting was used to prove that the zeros or poles of any non-zero meromorphic modular form of level~$N$ with algebraic Fourier coefficients are either transcendental or imaginary quadratic irrational.

Motivated by these findings and the work of Bringmann and Kane~\cite{BK}, we turn our attention to the divisor lifting with respect to sesquiharmonic Maass functions that are preimages of $J_n$ under the hyperbolic Laplace operator
\[
\Delta_0 = -y^2 \left( \frac{\partial^2}{\partial x^2} + \frac{\partial^2}{\partial y^2} \right).
\]

In a recent paper, Bringmann and Kane~\cite{BK} constructed a family of
sesquiharmonic Maass functions $\mathbb{J}_n$ satisfying
$\Delta_0(\mathbb{J}_n)=J_n$, where
\[
\mathbb{J}_0(\tau)=\log\!\bigl(y|\eta(\tau)|^4\bigr)+1.
\]
They further proved that if $f$ is a meromorphic modular function for $\G$
with constant term one in its Fourier expansion, then
\[
\langle J_n, \log|f| \rangle = -2\pi \sum_{z \in \G\backslash \mathbb{H}} \frac{\operatorname{ord}_z(f)}{\omega_z} \mathbb{J}_n(z).
\]
Here $\langle\cdot,\cdot\rangle$ denotes the regularized Petersson inner product; see \cite[Definition~2.1]{JKKM} for details.
Moreover, they extended this result by replacing the meromorphic function $f$ with a meromorphic modular form of weight~$k$ times $y^{\frac{k}{2}}$.

In view of our goal of extending the Hecke equivariance results to divisor liftings
associated with sesquiharmonic Maass forms, we are led to consider a different family
of sesquiharmonic Maass functions, which we denote by $\mathbb{J}_{1,n}$.
As in the work of Bringmann and Kane~\cite{BK}, these functions arise as sesquiharmonic preimages
of $J_n$ under $\Delta_0$.

Let ${\mathbb J}_{1,0}(\t):={\mathbb J}_{0}(\t)$ and 
$\J_{1,n}(\t):= {\rm Coeff}_{(s-1)^1} F_{1,0,-n}(\tau,s)$ for a positive integer $n$.
Here $F_{1,0,-n}(\tau,s)$ denotes the weight~$0$ Maass--Poincar\'e series of index $-n$
for $\G$, defined for ${\rm Re}(s)>1$ by
\[
F_{1,0,-n}(\tau,s)
:= 2\pi \sqrt{n}
\sum_{\gamma \in \Gamma_\infty \backslash \G}
\left( {\rm Im}(\gamma\tau) \right)^{1/2}
I_{s-\frac12}\!\left(2\pi n {\rm Im}(\gamma\tau)\right)
e\!\left(-n {\rm Re}(\gamma\tau)\right),
\]
where $I_\nu$ denotes the $I$--Bessel function.

For $n>0$, the functions $\mathbb{J}_{1,n}$ and $\mathbb{J}_n$ are related by an explicit logarithmic correction term.
More precisely, a direct computation carried out in \cite{JKKM} shows that
\begin{equation}\label{J-J1-relation}
	\mathbb{J}_{1,n}(\tau)
	= \mathbb{J}_{n}(\tau)
	+ 4\sigma_1(n)\log\!\bigl(v^6|\Delta(\tau)|\bigr)
	+ C_n,
\end{equation}
for some constant $C_n$.

Now we define the divisor lifting of a meromorphic modular form $f$ of weight~$k$ on $\G$ with respect to ${\mathbb J}_{1,n}$ by
\begin{equation}\label{Dlift}
\mathbb{D}(f):=\sum_{n=0}^\i \sum_{z\in \G\backslash \mathbb{H}}
\frac{\operatorname{ord}_z(f)}{\omega_z} \mathbb{J}_{1,n}(z) q^n.
\end{equation}

Bringmann and Kane~\cite{BK} defined a divisor lifting using the functions $\mathbb{J}_n$
in exactly the same way as in \eqref{Dlift}, and denoted it by $\mathbbm{f}^{\rm div}$.
In~\cite[Corollary~1.5]{BK}, they showed that this divisor lifting maps a meromorphic modular
function to the holomorphic part of the Fourier expansion of a weight~$2$ sesquiharmonic
Maass form.
Defining the divisor lifting instead with the functions $\mathbb{J}_{1,n}$ as in \eqref{Dlift},
the same conclusion continues to hold, which can be verified by following the argument of
\cite[Corollary~1.5]{BK}.


In this paper, we prove that this extended lifting is Hecke equivariant up to constant under the multiplicative Hecke operator on integral weight meromorphic modular forms and the Hecke operator on weight~$2$ sesquiharmonic Maass forms, similar to the original divisor lifting in \eqref{divmf}.

\begin{thm}\label{main-HE}
Let $f$ be a meromorphic modular form for $\G$.
Let $p$ be a prime and $\mathcal T(p)$ denote the multiplicative Hecke operator acting on the multiplicative group of integer weight meromorphic modular forms for $\G$, as defined in \eqref{mtp}.
Let $h := {\rm ord}_{\infty}(f)$ and assume that ${\rm Coeff}_{q^h}(f)=1$. Then
\begin{equation}\label{maineq2}
\mathbb D(f|\mathcal T(p))=(\mathbb D(f))|T_p -\left(\frac{k}{12}-h\right)(p-1) \log p.
\end{equation}
\end{thm}

In the authors’ previous work~\cite{JKK-hequiv}, it was shown that the divisor lifting $D(f)$ in \eqref{divmf} is exactly Hecke-equivalent, that is, $D(f|\mathcal T(p)) = D(f)|T_p$.
In contrast, Theorem~\ref{main-HE} shows that the divisor lifting $\mathbb{D}(f)$
defined in this paper is Hecke equivariant only up to an explicit additive constant.

The Hecke equivariance up to a constant established in Theorem~\ref{main-HE}
provides an extension of \eqref{hs} to sesquiharmonic Maass functions.
This leads us to the $p$-plication formula for the sesquiharmonic Maass functions:

\begin{cor}[$p$-plication formula]\label{app2}
Let $p$ be prime. Then for any positive integer $n$,
\begin{equation}\label{p-pli}
\J_{1,n}\mid T_p=\J_{1,pn}(\t)+p\,\J_{1,n/p}(\t)
\end{equation}
and
\begin{equation}\label{p-pli2}
\J_{1,0}\mid T_p=(p+1)\J_{1,0}-(p-1)\log p.
\end{equation}
\end{cor}

Utilizing the multiplicative property of the Hecke operator, Corollary~\ref{app2}
implies that the sesquiharmonic Maass functions $\J_{1,n}$ $(n>0)$ form a Hecke system.

\begin{cor}\label{app3}
Let $n$ be a positive integer. Then we have
\begin{equation}\label{J1Tn}
\J_{1,1}(\t)\mid T_n=\J_{1,n}(\t)
\end{equation}
and
\begin{equation}\label{J0Tn}
\J_{1,0}(\t)| T_n=\sigma(n)\,\J_{1,0}-\sum_{p^r\parallel n} r\,p^{r-1}(p-1)\log p.
\end{equation}
\end{cor}

Polyharmonic Maass forms, including sesquiharmonic Maass forms, have played important roles in number theory.
For a historical overview of the relationship between polyharmonic Maass forms and $L$-functions, refer to~\cite{TM-poly}.
In \cite{TM-trace}, Matsusaka elucidated the relationship between the twisted trace of CM values of the sesquiharmonic Maass function
$\mathbbm{f}(\t):=-\J_0(\t)+1$ and the value of the Dirichlet $L$-function at $s=1$.
Notably, $\mathbbm{f}(\t)=-\log(y|\eta(\t)|^4)$ appears in Kronecker's limit formula,
where $\eta(\t):= q^{-1/24}\prod_{n=1}^\i(1-q^n)$ is the Dedekind eta function.

Throughout this paper, let $\Delta\in\Z$ be a fundamental discriminant and $d$ be a positive integer such that $-d$ is a discriminant.
We denote by $\mathcal Q_{\Delta}$ the set of positive definite integral binary quadratic forms
$$
Q=[a,b,c]=aX^2+bXY+cY^2 \quad (a,b,c\in \mathbb Z)
$$
of discriminant $\Delta$, with the usual action of the group $\G$.
To each $Q\in \mathcal Q_{\Delta}$, we associate its unique root $\a_Q\in \mathbb H$, called a {\it CM point}.
The values assumed by the modular functions at $\a_Q$ are known as {\it singular moduli}, and they have important implications in number theory.

The genus character on $\mathcal Q_{d\Delta}\backslash\G$ is defined by
$$\chi_\Delta(Q):=\begin{cases}
(\frac{\Delta}{n}),\quad \text{if}\  \gcd(a,b,c,\Delta)=1\ \text{and}\ \gcd(n,\Delta)=1, \text{where}\ Q\ \text{represents}\ n,\\
0,\quad \text{otherwise}.
\end{cases}$$ 
For a $\G$-invariant function $f$, we define the twisted trace of singular moduli when $-\Delta d<0$ by
\begin{equation*}\label{tr}
\mathrm{Tr}_{\Delta,d}(f):=\sum_{Q\in\mathcal Q_{d\Delta}/\bar{\G}}\frac{\chi_\Delta(Q)}{\omega_{Q}}f(\a_Q),
\end{equation*}
where $\bar{\G}=\mathrm{PSL_2}(\mathbb Z)$ and $\omega_{Q}$ is the cardinality of the stabilizer of $Q$ in $\bar{\G}$.

As an application of the Hecke equivariance up to a constant established in Theorem~\ref{main-HE},
we obtain the following relation for twisted traces.

\begin{cor}\label{cong} Let $p$ be a prime not dividing a discriminant $\Delta>1$ and $m, n$ be non-negative integers with $n>0$. For a negative discriminant $-d$, we write $d=p^{2u}d'$ where $p^2\nmid d'$ and $u\geq 0$. Then when $\ell={\rm min}({\rm ord}_p(n),m)$, we have 
\begin{equation*} 
\displaystyle\sum_{t=0}^{\ell} p^{t}{\rm Tr}_{\Delta,d}\left(\J_{1,\frac{p^m n}{p^{2t}}}\right) =\begin{cases}\displaystyle \sum_{t=0}^m p^{m-t} Tr_{\Delta,p^{2u-2m+4t}d'}(\J_{1,n}),\hbox{ if } 0\leq m<u,\\ \displaystyle\sum_{t=0}^{m-u} \left(\frac{-d'}{p}\right)^{m-u-t}p^u Tr_{\Delta,p^{2t}d'}(\J_{1,n})+\displaystyle\sum_{t=1}^u p^{u-t} Tr_{\Delta,p^{2m-2u+4t}d'}(\J_{1,n}),\hbox{ if } m \geq u. 
\end{cases} 
\end{equation*} 
\end{cor}

Matsusaka~\cite[Corollary 1.4]{TM-trace} established a representation for the twisted trace of singular moduli associated with the sesquiharmonic Maass function $\mathbbm{f}$, connecting it to the trace of cycle integrals and the value of the Dirichlet $L$-function at $s=1$.
Specifically, if $\Delta>1$ and $-d$ is a fundamental discriminant, then
\begin{equation}\label{TMcor14}
\mathrm{Tr}_{\Delta,d}\left( -\log\left( y \left| \eta(\t) \right|^4 \right) \right)
= \sqrt{d}\, L_{-d}(1)\, \frac{h(\Delta)\log\varepsilon_{\Delta}}{\pi},
\end{equation}
where $L_D(s):=\sum_{n=1}^{\infty} \left( \frac{D}{n} \right) n^{-s}$ is the Dirichlet $L$-function, and $h(D)$ and $\varepsilon_D$ represent the narrow class number and fundamental unit of $\Q(\sqrt D)$, respectively.
Applying Theorem~\ref{main-HE} and the Hecke equivariance of the Borcherds isomorphism proved by Guerzhoy~\cite{Guer}, we can extend \eqref{TMcor14} to cases where $-d$ is not necessarily a fundamental discriminant.

\begin{cor}\label{app1} Suppose that $\Delta > 1$,
$-d=D\mathfrak{f}^2$ with a fundamental discriminant $D$ and $\gcd(D,\mathfrak{f})=1$, and $\mathfrak{f}=p_1p_2\cdots p_r$ is a prime factorization of $\mathfrak{f}$.
Then for any prime $p$ not dividing $d$, we have
$$\mathrm{Tr}_{\Delta,p^2d}\left( -\log\left( y \left| \eta(\t) \right|^4 \right) \right)=\prod_{i=1}^{r}\lt(p_i+1-\lt(\frac{D}{p_i}\rt)\rt)|D| L_D(1) \frac{h(\Delta)\log\varepsilon_{\Delta}}{\pi}.$$
\end{cor}

In the following section, we introduce sesquiharmonic Maass forms and discuss the Hecke equivariance of the Borcherds isomorphism, along with relevant results for later use. In Section~3, we demonstrate the Hecke equivariance of the divisor lifting with respect to sesquiharmonic Maass forms, proving Theorem~\ref{main-HE}. Sections~4 and~5 are dedicated to proving its applications as outlined in Corollaries~\ref{app2}, \ref{app3},\ref{cong}, and \ref{app1}, respectively.

\section{preliminaries}\label{pre}
\subsection{Sesquiharmonic Maass Forms}

For a matrix $\a=\sm a&b\\c&d\esm$ with a positive determinant, the action of the weight $k\in 2\Z$ slash operator on a function $f$ on $\H$ is given by
\begin{equation}\label{slash}
f|_k[\a]:=f\lt(\frac{a\t+b}{c\t+d}\rt)(c\t+d)^{-k}(\det\a)^{k/2}.
\end{equation}
A function $f$ is $\G$-invariant of weight $k$ if $f|_k \g = f$ for every $\g \in \G$.
When such a function is holomorphic on $\H$ and meromorphic at the cusp, it is called {\it weakly holomorphic}.
If a smooth $\G$-invariant function on $\H$ is annihilated by the weight $k$ hyperbolic Laplacian
$$
\Delta_k:=-y^2\lt(\frac{\partial^2}{\partial x^2}+\frac{\partial^2}{\partial y^2}\rt)
+ik y\lt(\frac{\partial}{\partial x}+i\frac{\partial}{\partial y}\rt),
$$
and grows at most linear exponentially at the cusp $i\i$, it is called a \textit{harmonic weak Maass form}.
If instead of being zero, $\Delta_k f$ is a non-zero weakly holomorphic modular form, it is referred to as a \textit{sesquiharmonic Maass form}.
In other words, if $f$ is a sesquiharmonic Maass form, then $\xi_k\circ \Delta_k(f)=0$,
where $\displaystyle{\xi_k:=2iy^{k}\frac{\overline{\partial }}{\partial \bar{\t}}}$ is the anti-linear differential operator.
Note that it holds $\Delta_k=-\xi_{2-k}\circ\xi_k$.
More generally, a \textit{polyharmonic (weak) Maass form} is a smooth $\G$-invariant function on $\H$ that satisfies the growth condition at the cusp and, upon repeated application of the $\Delta_k$ operator, eventually becomes zero.
Polyharmonic Maass forms that have moderate growth at the cusp were introduced by Lagarias--Rhoades in \cite{LR}, where they constructed a basis for the space of polyharmonic Maass forms.
Matsusaka later constructed a basis for the space of polyharmonic weak Maass forms in \cite{TM-poly}.

In \cite{JKK-sesqui}, the authors constructed a family of sesquiharmonic Maass functions, denoted by $\hat{J}_n$, which satisfy $\Delta_0(\hat{J}_n) = -J_n-24\sigma(n)$, and developed the generating function of traces of their cycle integrals.
In their recent work~\cite{BK}, Bringmann and Kane introduced another family of sesquiharmonic Maass functions, $\mathbb{J}_n$, where $\mathbb{J}_0(\t)=\log(y|\eta(\t)|^4)+1$.
These functions satisfy $\Delta_0(\mathbb{J}_n) = J_n$, and they provided the Fourier expansion for $\mathbb{J}_n$.
To state Fourier expansions of sesquiharmonic Maass forms, we first recall special values of Whittaker functions as given in \cite[Lemma 2.1]{BK}.
For $w \in\C$ and $\textrm{Re}(s)>0$, the incomplete Gamma function is defined by 
$\displaystyle{\G(s,w):=\int_w^\i e^{-t}t^s\frac{dt}{t}}$.
For $\ka\in\Z$, the Whittaker function is given by
$$
W_{\ka}(w): = 
\begin{cases}
\Gamma(1-\kappa, -2w) + \frac{(-1)^{1-\kappa} \pi i}{(\kappa - 1)!}, &  \text{if}\ w>0,\\
\Gamma(1-\kappa, -2w), & \text{if}\ w<0.
\end{cases}
$$
Additionally, let
\[
\mathcal W_s(w) := \int^w_{\text{sgn}(w)\infty} W_{2-s}(-t) t^{-s} e^{2t} \, dt.
\]

The Fourier expansion of a sesquiharmonic Maass form can be described as follows:
\begin{prop}\cite[Lemma 4.1]{BK}\label{lem:4.1}
If $\mathbbm h$ is a sesquiharmonic Maass form of weight $\kappa \in \mathbb{Z}\setminus\{1\}$, then for $y \gg 0$ we have $\mathbbm h = \mathbbm h^{++} + \mathbbm h^{+-} + \mathbbm h^{--}$, where
\begin{align*}
\mathbbm h^{++}(\t)& := \sum_{m \gg -\infty} c^{++}(m) q^m,\\
\mathbbm h^{+-}(\t)& := c^{+-}(0) y^{1-\kappa} + \sum_{\substack{m \ll \infty \\ m \neq 0}} c^{+-}(m) \mathcal W_{\kappa}(2\pi my) q^m,\\
\mathbbm h^{--}(\t)& := c^{--}(0) \log(y) + \sum_{\substack{m \gg -\infty \\ m \neq 0}} c^{--}(m) \mathcal W_{\kappa}(2\pi my) q^m.
\end{align*}
\end{prop}
Combining \eqref{J-J1-relation} with \cite[Lemma 4.1]{BK}, we obtain the Fourier expansion of $\mathbb{J}_{1,n}$.
More precisely, we have the following result.

\begin{prop}\label{prop:J1-expansion}
There exist constants $c^{++}_{1}(m), c^{+-}_{1}(m)$, and $c^{--}_{1}(m)\in\C$
such that
\begin{align*}
\mathbb{J}_{1,n}(\tau)
&= \delta_{n=0}
+ \sum_{m \geq 1} c^{++}_{1}(m) q^m
+ c^{+-}_{1}(0) y
+ \sum_{m \leq -1} c^{+-}_{1}(m) \mathcal W_0(2\pi m y) q^m \\
&\quad
+ \delta_{n=0}\log(y)
+ 2\delta_{n\neq 0}\mathcal W_0(-2\pi n y) q^{-n}
+ \sum_{m \geq 1} c^{--}_{1}(m)\mathcal W_0(2\pi m y) q^m,
\end{align*}
where $\delta_S:=1$ if some statement $S$ is true and $0$ otherwise.
\end{prop}

As stated in \cite[Lemma 2.2]{BK}, direct computation shows that for $m \in \mathbb{Z} \setminus \{0\}$,  
\begin{align*}
\xi_{\kappa}(W_{\kappa}(2\pi my)q^m) &= -(-4\pi m)^{1-\kappa} q^{-m}\\
\xi_{\kappa}(\mathcal W_{\kappa}(2\pi my)q^m) &= (2\pi m)^{1-\kappa} W_{2-\kappa}(-2\pi my) q^{-m}.
\end{align*}
Hence, $W_{\kappa}(2\pi my) q^m$ is annihilated by $\Delta_{\kappa}$, and $\mathcal W_{\kappa}(2\pi my) q^m$ is annihilated by $\xi_{\kappa} \circ \Delta_{\kappa}$. Therefore, $\mathbbm h$ is harmonic if and only if $\mathbbm h^{--}(\t) = 0$ and weakly holomorphic  if and only if $\mathbbm h^{+-}(\t) = 0$ and $\mathbbm h^{--}(\t) = 0$. Thus, it is appropriate, as indicated  in \cite{BK}, to call $\mathbbm h^{++}(\t)$ the holomorphic part, $\mathbbm h^{+-}(\t)$ the harmonic part, and $\mathbbm h^{--}(\t)$ the sesquiharmonic part of $\mathbbm h(\t)$. 
%

%

\subsection{Hecke operators}

Throughout this section, we let $f$ be a $\Gamma$-invariant function of weight $k\in 2\Z$ with Fourier expansion
$$
f(\t)=\sum_{n\in \Z} a(n,y) q^n.
$$
For a positive integer $m$, consider the set
$$
\mathcal A_m:=\{\a=\sm a&b\\c&d\esm \mid a,b,c,d\in\Z,\ \det\a=m\}
$$
and let $\alpha_1,\alpha_2,\dots,\alpha_s$ denote a set of right coset representatives for the action of $\G$ on $\mathcal A_m$.
The $m$-th normalized Hecke operator $T_m$ acting on $f$ is defined by
\begin{equation}\label{hp}
f(\t)| T_m:=\sum_{\alpha\in\G\backslash\mathcal A_m} f|_k[\alpha_i].
\end{equation}
In the case when $m=p$ is prime, its action on Fourier expansions can be described as
\begin{equation}\label{hpupvp}
f(\t)| T_p
= p^{-k/2}\sum_{i=0}^{p-1} f\!\left(\frac{\t+i}{p}\right)+p^{k/2}f(p\t).
\end{equation}

Hecke operators form a commutative algebra and satisfy the following multiplicative property:
\begin{prop}\label{Hecke2} We have, for any distinct primes $p$ and $\ell$ and positive integers $r$ and $t$,
\begin{enumerate}
\item[(1)] $T_{p^r}\circ T_{\ell^t}=T_{\ell^t}\circ T_{p^r}$,
\item[(2)] $T_{p^r}\circ T_{p}=T_{p^{r+1}} + pT_{p^{r-1}}$.
\end{enumerate}
\end{prop}

We next define a multiplicative Hecke operator acting on $\mathcal M(\G)$, the multiplicative group of integer weight meromorphic modular forms for $\G$.
We let $\mathcal M_{k}(\G)\subset\mathcal M(\G)$ denote the subset that consists of modular forms of weight $k$.
 
Let $p$ be prime and $f\in \mathcal M_{k}(\G)$.  Following Guerzhoy \cite{Guer}, we define the multiplicative Hecke operator $\mathcal T(p)$ acting on $\mathcal M(\G)$ by
\begin{equation}\label{mtp}f|\mathcal T(p):=f|\mathcal T_{k}(p):=\varepsilon p^{k(p-1)/2}\prod_{i=1}^{p+1}f|_k[\alpha_i],\end{equation}
where $\varepsilon$ is a constant chosen so that the leading coefficient of $f|\mathcal T(p)$ is $1$ and  $\alpha_i$ are the right coset representatives of the action of $\G$ for the set of $2\times 2$ matrices with integer entries and determinant $p$. 
Then $\mathcal T(p)$ maps elements of $\mathcal M_{k}(\G)$ to elements of $\mathcal M_{(p+1)k}(\G)$. 

\subsection{Hecke equivariance of the Borcherds Isomorphism}

Let $M^!_{1/2}(4)$ denote the space of weakly holomorphic modular forms of weight $1/2$ for $\G_0(4)$ which satisfy Kohnen's plus-space condition. 
The basis elements of $M^{!}_{1/2}(4)$ found by Zagier are given by
$$f_d=q^{-d}+\sum_{\substack{0<n \equiv0,1\ (4)}}A(n,d)q^n$$
for all negative discriminants $-d$. We let $\Delta>1$.  By \cite[Theorem 6.1]{BO}, the generalized Borcherds product of $f_d$,
$$\MB(f_d)=\Psi_\Delta(\t,f_d):=\Psi_{\Delta,r}(\t,f_d)=\prod_{n=1}^\i\prod_{b(\Delta)}\lt(1-e^{2\pi i \lt(\frac{b}{\Delta}\rt)}q^n\rt)^{(\frac{\Delta}{b})A(\Delta n^2,d)}$$
is a weight $0$ meromorphic modular form for the group $\G$ with a unitary character. 
It was shown in (\cite[Eq. (8.2)]{BO}) that 
\begin{equation}\label{bp}\MB(f_d)
=\prod_{Q\in \mathcal Q_{\Delta d}/\G}(j(\t)-j(\alpha_Q))^{\chi_\Delta(Q)/|\bar{\G}_Q|},
\end{equation}
where $\G_Q$ is the stabilizer of $Q$ in ${\G}$.
Hence we deduce that
\begin{equation*}\label{tr}(\mathbb D\circ \MB)(f_d)=-\sum_{n=0}^\i \lt(\sum_{Q\in\mathcal Q_{d\Delta}/\G}\frac{\chi_\Delta(Q)}{|\bar{\G}_Q|}\J_{1,n}(\a_Q)\rt)q^n.
\end{equation*}
Thus we may write  the composition $\mathbb D\circ \MB$ lift of $f_d$ as the generating function of the twisted trace of singular moduli:
\begin{equation*}\label{gtr}(\mathbb D\circ \MB)(f_d)=-\sum_{n=0}^\i Tr_{\Delta,d}(\J_{1,n})q^n.
\end{equation*}
 Guerzhoy \cite{Guer} proved that the Borcherds isomorphism is Hecke equivariant under the multiplicative Hecke operator acting on the integral weight meromorphic modular forms and the usual Hecke operator 
$ pT_{1/2}(p^2)$
acting on the half integral weight forms. That is,
\begin{equation}\label{gbhe}
\MB(\phi)|\mathcal T(p)=\MB(\phi|pT_{1/2}(p^2)).
\end{equation}

\section {Proof of Theorem \ref{main-HE}}\label{Proofmain}

We first treat the case $h=0$.
It is proved in \cite[Corollary 4.6 and (4.4)]{JKKM} that
\begin{equation}\label{inndiv}
\lt\langle j_{1,n}, \log(y^{k/2}|f|)\rt\rangle=
\begin{cases}
-2\pi\sum_{z\in \G\backslash \mathbb{H}}\frac{\hbox{ord}_{z}(f)}{\omega_{z}}\mathbb{J}_{1,n}(z),  &  \hbox{ if } n>0 \\
-2\pi\sum_{z\in \G\backslash \mathbb{H}}\frac{\hbox{ord}_{z}(f)}{\omega_{z}}\mathbb{J}_{1,0}(z)+\frac{\pi}{3}kc,  &  \hbox{ if } n=0. \\
\end{cases}
\end{equation}
where 
$j_{1,n}=F_{1,0,-n}(\tau,1)$ for $n>0$, $j_{1,0}=1$ and
$c:= 1-12 \zeta'(-1) - \log (4\pi)$. 

Now we set
\begin{align*}
\hat{\mathbb{D}}(f)&:=\sum_{n=0}^\i a^{++}(n)q^n+\sum_{n\ll \i}^\i a^{+-}(n,y)q^n+\sum_{n\gg -\i}^\i a^{--}(n,y)q^n,\\
\hat{\mathbb{D}}(f|\mathcal T(p))&:=\sum_{n=0}^\i b^{++}(n)q^n+\sum_{n\ll \i}^\i b^{+-}(n,y)q^n+\sum_{n\gg -\i}^\i b^{--}(n,y)q^n.
\end{align*}
Then we obtain by \eqref{inndiv} that for $n>0$, 
\begin{align*}
a^{++}(n)&=-\frac{1}{2\pi}\lt\langle j_{1,n}, \log(y^{k/2}|f|)\rt\rangle \ \textrm{and}\\
b^{++}(n)&=-\frac{1}{2\pi}\lt\langle j_{1,n}, \log(y^{k(p+1)/2}|f|\mathcal T(p)|)\rt\rangle
\end{align*}
and for $n=0$,
\begin{align*}
a^{++}(0)&=-\frac{1}{2\pi}\lt\langle j_{1,0}, \log(y^{k/2}|f|)\rt\rangle
+\frac{kc}{6}
 \ \textrm{and}\\
b^{++}(0)&=-\frac{1}{2\pi}\lt\langle j_{1,0}, \log(y^{k(p+1)/2}|f|\mathcal T(p)|)\rt\rangle
+\frac{k(p+1)c}{6}.
\end{align*}

Considering \eqref{hpupvp}, we prove \eqref{maineq2} in the case $h=0$ by establishing 
\begin{equation}\label{goal1}
b^{++}(n)=
\begin{cases} 
a^{++}(pn)+pa^{++}(n/p), 
& \hbox{ if  } n>0 \\
(p+1) a^{++}(0) - \frac{k}{12}(p-1) \log p &
\hbox{ if } n=0.
\end{cases} 
\end{equation}
For simplicity, we denote $\mathcal Tf:=f|\mathcal T(p)$. Note that  $\mathcal Tf$ has weight $k(p+1)$, while $\log(y^{k/2}|f|)$ has weight 0 for $\G$. 
Then by the definition of the action of the multiplicative Hecke operator, we find that  
\begin{align}\label{inn1}
\lt\langle j_{1,n}, \log(y^{k(p+1)/2}|\mathcal Tf|)\rt\rangle&=\lt\langle j_{1,n}, \log\lt(y^{k(p+1)/2}\lt|\varepsilon p^{k(p-1)/2}f|_k\mat p&0\\0&1\emat\prod_{i=0}^{p-1}f|_k\mat 1&i\\0&p\emat\rt|\rt)\rt\rangle
\notag \\
&=\lt\langle j_{1,n}, \log\lt(y^{k/2}|f(p\t)|\rt)\rt\rangle+\sum_{i=0}^{p-1}\lt\langle j_{1,n}, \log\lt(y^{k/2}|f(\frac{\t+i}{p})|\rt)\rt\rangle \notag \\
&=\lt\langle j_{1,n}, \log\lt((py)^{k/2}|f(p\t)|\rt)\rt\rangle+\sum_{i=0}^{p-1}\lt\langle j_{1,n}, \log\lt((y/p)^{k/2}|f(\frac{\t+i}{p})|\rt)\rt\rangle\notag\\
&\quad+\lt\langle j_{1,n}, -k/2\log p\rt\rangle+p\lt\langle j_{1,n}, k/2\log p\rt\rangle\notag\\
&=\lt\langle j_{1,n}, \lt(\log y^{k/2}|f(\t)|\rt)|{T_p}\rt\rangle+\frac{(p-1)k}{2}\log p \langle j_{1,n},1\rangle.
\end{align}
Hence we observe that for $n>0$,
\begin{align*}
b^{++}(n)&=-\frac{1}{2\pi}\lt\langle j_{1,n}, \log(y^{k(p+1)/2}|\mathcal{T}f|)\rt\rangle\\
&=-\frac{1}{2\pi}\lt\langle j_{1,n}, \lt(\log y^{k/2}|f(\t)|\rt)|T_p\rt\rangle-(p-1)\frac{k}{4\pi}\log p \langle j_{1,n},1\rangle
\hbox{ \, \, by \eqref{inn1}} \\
&=-\frac{1}{2\pi}\lt\langle j_{1,n}|T_p, \lt(\log y^{k/2}|f(\t)|\rt)\rt\rangle
-(p-1)\frac{k}{4\pi}\log p \langle j_{1,n},1\rangle
 \\
&=-\frac{1}{2\pi}\lt\langle j_{1,pn}+pj_{1,n/p}, \lt(\log y^{k/2}|f(\t)|\rt)\rt\rangle
-(p-1)\frac{k}{4\pi}\log p \langle j_{1,n},1\rangle
 \\
&= a^{++}(pn)+pa^{++}(n/p)
-(p-1)\frac{k}{4\pi}\log p \langle j_{1,n},1\rangle 
\\
&= a^{++}(pn)+pa^{++}(n/p) 
\hbox{ since } \langle j_{1,n},1\rangle =0 
\hbox{ by \cite[(1.5)]{DIT} }.
\end{align*}
Here, the third equality follows from the fact that the Hecke operator $T_p$ is self-adjoint for the regularized Petersson inner product.

In the case $n=0$, we have
\begin{align*}
b^{++}(0)&=-\frac{1}{2\pi}\lt\langle j_{1,0}, \log(y^{k(p+1)/2}|\mathcal{T}f|)\rt\rangle 
+ \frac{k(p+1)c}{6} \\
&=-\frac{1}{2\pi}\lt\langle j_{1,0}, \lt(\log y^{k/2}|f(\t)|\rt)|T_p\rt\rangle-(p-1)\frac{k}{4\pi}\log p \langle j_{1,0},1\rangle
+ \frac{k(p+1)c}{6}
\hbox{ \, \, by \eqref{inn1}} \\
&=-\frac{1}{2\pi}\lt\langle j_{1,0}|T_p, \lt(\log y^{k/2}|f(\t)|\rt)\rt\rangle
-(p-1)\frac{k}{4\pi}\log p \langle j_{1,0},1\rangle
 + \frac{k(p+1)c}{6} \\
&=-\frac{1}{2\pi}\lt\langle (p+1)j_{1,0}, \lt(\log y^{k/2}|f(\t)|\rt)\rt\rangle
-(p-1)\frac{k}{4\pi}\log p \langle j_{1,0},1\rangle
+ \frac{k(p+1)c}{6}
 \\
&= (p+1) a^{++}(0)
-(p-1)\frac{k}{4\pi}\log p \langle j_{1,0},1\rangle 
\\
&= (p+1) a^{++}(0) - \frac{k}{12} (p-1) \log p
\hbox{ since } \langle j_{1,0},1\rangle =\langle 1,1\rangle =\hbox{vol}(\G\backslash \H)=\frac{\pi}{3} .
\end{align*}
This proves \eqref{goal1}, and thus \eqref{maineq2} holds when $h=0$.

We now consider the general case. Since ${\rm ord}_{\infty} f=h$ and 
${\rm Coeff}_{q^h} f=1$, we can write $f=q^h\prod_{n=1}^\i(1-q^n)^{\ell(n)}.$ 
If we define $\displaystyle{g(\t)=\frac{f(\t)}{\Delta^h(\t)}}$, then applying \eqref{maineq2} in the case $h=0$ to $g$, we obtain 
\begin{equation}\label{tp}\mathbb D(g|\mathcal T(p))=\mathbb D(g)|T_p-\left(\frac{k}{12}-h\right) (p-1) \log p.
\end{equation} 
The left-hand side of \eqref{tp} is equal to
$$\mathbb D\lt(\frac{f}{\Delta^h}|\mathcal T(p)\rt)=\mathbb D\lt(\frac{f|\mathcal T(p)}{(\Delta|\mathcal T(p))^{h}}\rt)=\mathbb D\lt(\frac{f|\mathcal T(p)}{\Delta^{(p+1)h}}\rt)=\mathbb D(f|\mathcal T(p))-(p+1)h\mathbb D(\Delta)=\mathbb D(f|\mathcal T(p)),$$
and the right-hand side of \eqref{tp} is equal to
\begin{align*}
\mathbb D\lt(\frac{f}{\Delta^h}\rt)|T_p -\left(\frac{k}{12}-h\right) (p-1) \log p
&=\lt(\mathbb D(f)-h\mathbb D(\Delta)\rt)|T_p -\left(\frac{k}{12}-h\right) (p-1) \log p \\
&=\mathbb D(f)|T_p-\left(\frac{k}{12}-h\right) (p-1) \log p.
\end{align*}
Comparing the above identities, we obtain \eqref{maineq2}.

\section{Proof of Corollaries \ref{app2} and \ref{app3}}
\begin{proof}[Proof of Corollary \ref{app2}]
Let $z\in\mathbb H$ be arbitrary.  Then $$\mathbb D(j(\t)-j(z))=\sum_{n=0}^\i \J_{1,n}(z)q^n.$$
Hence, for a prime $p$, 
$$
\mathbb D(j(\t)-j(z))|T_p=\sum_{n=0}^\i \lt(\J_{1,pn}(z)+p\J_{1,n/p}(z)\rt)q^n.
$$
On the other hand, if we let $\Phi_m(X,Y)\in\Z[X,Y]$ denote the $m$-th modular polynomial which is defined by the property
$$\Phi_m(X,j(\t)):=\prod_{\alpha\in\G\backslash \mathcal A_m}(X-j(\alpha\circ \t)),$$
then by the symmetry of $\Phi_n(X,Y)$, we obtain 
\begin{align*}
(j(\t)-j(z))|\mathcal T(p)&=(j(z)-j(p\t))(j(z)-j(\t/p))\cdots(j(z)-j((\t+p-1)/p))\\
&=(j(\t)-j(pz))(j(\t)-j(z/p))\cdots(j(\t)-j((z+p-1)/p)).
\end{align*}
Thus 
\begin{align*}\mathbb D((j(\t)-j(z)|\mathcal T(p))&=\sum_{n=0}^\i \J_{1,n}(pz)q^n+\sum_{n=0}^\i \J_{1,n}\lt(\frac{z}{p}\rt)q^n+\cdots+\sum_{n=0}^\i \J_{1,n}\lt(\frac{z+p-1}{p}\rt)q^n\\
&=\sum_{n=0}^\i \lt(\J_{1,n}(z)|T_p\rt) q^n.
\end{align*}
Corollary \ref{app2} is now valid as a result of Theorem \ref{main-HE}.
\end{proof}

\begin{proof}[Proof of Corollary \ref{app3}]
{\rm (i)} In this section, we let  $p$ be a prime and  $r\in \mathbb N$. 
We begin by proving the equation 
\begin{equation}\label{cc}
\J_{1,1} | T_{p^r} = \J_{1,p^r}
\end{equation} via induction on $r$.  For the case when $r=1$, the equality holds by Corollary \ref{app2}. The inductive step follows from Proposition \ref{Hecke2} (2) and Corollary \ref{app2} as follows:
\begin{align*}
\J_{1,1} | T_{p^{r+1}}& = (\J_{1,1} | T_{p^r})|T_p - p \J_{1,1} | T_{p^{r-1}} \\
& = \J_{1,p^r}| T_p - p \J_{1,p^{r-1}}=\J_{1,p^{r+1}}+p\J_{1,p^{r-1}} - p\J_{1,p^{r-1}}= \J_{1,p^{r+1}}.
\end{align*}
With the same argument, we can show that for any distinct primes $p$ and $q$ and $r,s\in\mathbb N$,
\begin{equation}\label{cc1}
\J_{1,p^r} | T_{q^s} = \J_{1,p^rq^s}.
\end{equation}
Now if $n=p_1^{r_1}p_2^{r_2}\cdots p_t^{r_t}$, where $p_1, \ldots, p_t$ are distinct primes, then by applying Proposition~\ref{Hecke2}~(1), \eqref{cc}, and \eqref{cc1} in turn, we establish Corollary  \ref{app3}:
\begin{align*}
\J_{1,1} | T_{n} &= (\J_{1,1} | T_{p_1^{r_1}})|T_{p_2^{r_2}} \cdots T_{p_t^{r_t}}
=(\J_{1,p_1^{r_1}})|T_{p_2^{r_2}} \cdots T_{p_t^{r_t}}\\
& =(\J_{1,1} | T_{p_1^{r_1}p_2^{r_2}})|T_{p_3^{r_3}} \cdots T_{p_t^{r_t}}=\cdots=\J_{1,n}.
\end{align*}

\noindent {\rm (ii)} We first compute the action of $T_{p^r}$ by induction on $r$.
Using the Hecke relation
\[
T_{p^{r+1}}=T_{p^r}T_p-p\,T_{p^{r-1}},
\]
together with \eqref{p-pli2}, a straightforward induction yields
\begin{equation}\label{tpr}
\J_{1,0}\,\big|\,T_{p^r}
=
(1+p+\cdots+p^r)\,\J_{1,0}
-
r\,p^{r-1}(p-1)\log p.
\end{equation}

From equation~\eqref{tpr} together with the multiplicativity of the Hecke operators, the result follows.
\end{proof}

\section{Proof of Corollaries \ref{cong} and \ref{app1}}

Let $\mathbb M^{\textrm{s}}_2(\G)$ denote the space of the holomorphic parts of weight~$2$ sesquiharmonic Maass forms.
By \eqref{gbhe}, the left square in the diagram below is commutative
and 
by Theorem \ref{main-HE}, the right square is commutative
up to the addition of constant $-\left( \frac{k}{12}-h\right)(p-1)\log p$:
$$\begin{tikzcd}
M^!_{1/2}(4) \arrow[r,"\MB"] \arrow[d, "pT_{1/2}(p^2)"] & \mathcal M(\G) \arrow[r,"\mathbb D"] \arrow[d, "\mathcal{T}(p)"] & \mathbb M^{\textrm{s}}_2(\G) \arrow[d, "T_p"] \\
M^!_{1/2}(4)  \arrow[r,"\MB"] & \mathcal M(\G) \arrow[r,"\mathbb D"] & \mathbb M^{\textrm{s}}_2(\G)  \\
\end{tikzcd}$$

The identical approach employed to establish Theorem 4.1 in \cite{JKK-hequiv} can be utilized to obtain Corollary~\ref{cong}.


In particular, when $\Delta > 1$, $p$ is a prime not dividing $d$, and $d$ is square-free, we have
\begin{equation}\label{psquare}
(p+1)\mathrm{Tr}_{\Delta,d}(\J_{1,0})=\lt(\frac{-d}{p}\rt)\mathrm{Tr}_{\Delta,d}(\J_{1,0})+\mathrm{Tr}_{\Delta,p^2d}(\J_{1,0}).
\end{equation}
According to \cite[Theorem 4.1]{JKK-hequiv}, we also have
\begin{equation}\label{psquare0}
(p+1)\mathrm{Tr}_{\Delta,d}(1)=\lt(\frac{-d}{p}\rt)\mathrm{Tr}_{\Delta,d}(1)+\mathrm{Tr}_{\Delta,p^2d}(1).
\end{equation}
Given that $\mathbbm{f}(\t)=-\log\left( y \left| \eta(\t) \right|^4\rt)=-\J_{1,0}(\t)+1$, from \eqref{psquare} and \eqref{psquare0}, we deduce that
\begin{equation}\label{psquare2}
\lt( p+1- \lt(\frac{-d}{p}\rt)\rt)\mathrm{Tr}_{\Delta,d}\lt(-\log\left( y \left| \eta(\t) \right|^4\rt)\rt)=\mathrm{Tr}_{\Delta,p^2d}\lt(-\log\left( y \left| \eta(\t) \right|^4\rt)\rt).
\end{equation}
Therefore, if $-d$ is a fundamental discriminant, then by \eqref{TMcor14}, we have 
\begin{equation}\label{psquare3}
\mathrm{Tr}_{\Delta,p^2d}\lt(-\log\left( y \left| \eta(\t) \right|^4\rt)\rt)=\lt( p+1- \lt(\frac{-d}{p}\rt)\rt)\sqrt{d} L_{-d}(1)\frac{h(\Delta)\log\varepsilon_{\Delta}}{\pi}.
\end{equation}
To complete the proof, suppose that $-d=D\mathfrak{f}^2$ with a fundamental discriminant $D$ and $\gcd(D,\mathfrak{f})=1$, and $\mathfrak f$ has the prime factorization $\mathfrak f=p_1p_2\cdots p_r$.
Using \eqref{psquare3} inductively, we obtain Corollary \ref{app1}.

For example, let $\Delta=5$, $-d=-4$ and $p=3$. We can then compute the following using \eqref{TMcor14}:
$$\mathrm{Tr}_{5,4}(\J_{1,0})=\sqrt{4} L_{-4}(1) \frac{h(5)\log\varepsilon_5}{\pi}=\log(1+\sqrt2).$$
Consequently, from Corollary \ref{app1}, we obtain 
 $$\mathrm{Tr}_{5,36}(\J_{1,0})=\lt(4-\lt(\frac{-4}{3}\rt)\rt)\mathrm{Tr}_{5,4}(\J_{1,0})=5\log(1+\sqrt2).$$

\end{document}